\documentclass[a4paper,10pt]{amsart}

\usepackage[utf8]{inputenc}
\usepackage{url}

\usepackage[expansion=false]{microtype}
\usepackage{amssymb,amsmath,amsopn,amsxtra,amsthm,amsfonts}
\usepackage{xr}
\usepackage[mathcal]{euscript}
\usepackage{mathalfa}

\usepackage[english, status=final, nomargin, inline]{fixme}
\fxusetheme{color}

\usepackage[pdfusetitle,unicode,hidelinks]{hyperref}

\usepackage{enumitem}
\setlist[enumerate,1]{label={(\arabic*)},itemsep=\parskip} 
\setlist[itemize,1]{itemsep=\parskip} 
\newlist{thmlist}{enumerate}{2}
\setlist[thmlist,1]{label={\em(\roman*)},ref={(\roman*)},%
  itemsep=\parskip,leftmargin=*,align=left}
\setlist[thmlist,2]{label={\em(\alph*)},ref={(\alph*)},%
  itemsep=\parskip,leftmargin=*,align=left,topsep=0.1cm}
\newlist{remlist}{enumerate}{2}
\setlist[remlist,1]{label={(\roman*)},ref={(\roman*)},itemsep=\parskip,%
  leftmargin=*,align=left}
\setlist[remlist,2]{label={(\alph*)},ref={(\alph*)},itemsep=\parskip,%
  leftmargin=*,align=left,topsep=0.1cm}

\makeatletter
\let\c@equation\c@subsubsection

\makeatother

\newtheorem{cor}{Corollary}
\newtheorem{lem}[cor]{Lemma}
\newtheorem{prop}[cor]{Proposition}

\newtheorem{thm}[cor]{Theorem}

\newtheorem*{claim*}{Claim}

\theoremstyle{definition}

\newtheorem{rem}[cor]{Remark}
\newtheorem{constr}[cor]{Construction}

\renewcommand{\eqref}[1]{(\ref{#1})}

\usepackage{tikz}
\usetikzlibrary{matrix}
\usepackage{tikz-cd}

\usepackage{xspace}

\usepackage{xifthen}

\usepackage{xparse}

\newcommand{\nc}{\newcommand}
\nc{\renc}{\renewcommand}
\nc{\ssec}{\subsection}
\nc{\sssec}{\subsubsection}
\nc{\on}{\operatorname}
\nc{\term}[1]{#1\xspace}

\newcommand{\p}{\mathfrak{p}}

\DeclareMathSymbol{A}{\mathalpha}{operators}{`A}
\DeclareMathSymbol{B}{\mathalpha}{operators}{`B}
\DeclareMathSymbol{C}{\mathalpha}{operators}{`C}
\DeclareMathSymbol{D}{\mathalpha}{operators}{`D}
\DeclareMathSymbol{E}{\mathalpha}{operators}{`E}
\DeclareMathSymbol{F}{\mathalpha}{operators}{`F}
\DeclareMathSymbol{G}{\mathalpha}{operators}{`G}
\DeclareMathSymbol{H}{\mathalpha}{operators}{`H}
\DeclareMathSymbol{I}{\mathalpha}{operators}{`I}
\DeclareMathSymbol{J}{\mathalpha}{operators}{`J}
\DeclareMathSymbol{K}{\mathalpha}{operators}{`K}
\DeclareMathSymbol{L}{\mathalpha}{operators}{`L}
\DeclareMathSymbol{M}{\mathalpha}{operators}{`M}
\DeclareMathSymbol{N}{\mathalpha}{operators}{`N}
\DeclareMathSymbol{O}{\mathalpha}{operators}{`O}
\DeclareMathSymbol{P}{\mathalpha}{operators}{`P}
\DeclareMathSymbol{Q}{\mathalpha}{operators}{`Q}
\DeclareMathSymbol{R}{\mathalpha}{operators}{`R}
\DeclareMathSymbol{S}{\mathalpha}{operators}{`S}
\DeclareMathSymbol{T}{\mathalpha}{operators}{`T}
\DeclareMathSymbol{U}{\mathalpha}{operators}{`U}
\DeclareMathSymbol{V}{\mathalpha}{operators}{`V}
\DeclareMathSymbol{W}{\mathalpha}{operators}{`W}
\DeclareMathSymbol{X}{\mathalpha}{operators}{`X}
\DeclareMathSymbol{Y}{\mathalpha}{operators}{`Y}
\DeclareMathSymbol{Z}{\mathalpha}{operators}{`Z}

\nc{\sA}{\ensuremath{\mathcal{A}}\xspace}
\nc{\sB}{\ensuremath{\mathcal{B}}\xspace}
\nc{\sC}{\ensuremath{\mathcal{C}}\xspace}
\nc{\sD}{\ensuremath{\mathcal{D}}\xspace}
\nc{\sE}{\ensuremath{\mathcal{E}}\xspace}
\nc{\sF}{\ensuremath{\mathcal{F}}\xspace}
\nc{\sG}{\ensuremath{\mathcal{G}}\xspace}
\nc{\sH}{\ensuremath{\mathcal{H}}\xspace}
\nc{\sI}{\ensuremath{\mathcal{I}}\xspace}
\nc{\sJ}{\ensuremath{\mathcal{J}}\xspace}
\nc{\sK}{\ensuremath{\mathcal{K}}\xspace}
\nc{\sL}{\ensuremath{\mathcal{L}}\xspace}
\nc{\sM}{\ensuremath{\mathcal{M}}\xspace}
\nc{\sN}{\ensuremath{\mathcal{N}}\xspace}
\nc{\sO}{\ensuremath{\mathcal{O}}\xspace}
\nc{\sP}{\ensuremath{\mathcal{P}}\xspace}
\nc{\sQ}{\ensuremath{\mathcal{Q}}\xspace}
\nc{\sR}{\ensuremath{\mathcal{R}}\xspace}
\nc{\sS}{\ensuremath{\mathcal{S}}\xspace}
\nc{\sT}{\ensuremath{\mathcal{T}}\xspace}
\nc{\sU}{\ensuremath{\mathcal{U}}\xspace}
\nc{\sV}{\ensuremath{\mathcal{V}}\xspace}
\nc{\sW}{\ensuremath{\mathcal{W}}\xspace}
\nc{\sX}{\ensuremath{\mathcal{X}}\xspace}
\nc{\sY}{\ensuremath{\mathcal{Y}}\xspace}
\nc{\sZ}{\ensuremath{\mathcal{Z}}\xspace}

\nc{\bA}{\ensuremath{\mathbf{A}}\xspace}
\nc{\bB}{\ensuremath{\mathbf{B}}\xspace}
\nc{\bC}{\ensuremath{\mathbf{C}}\xspace}
\nc{\bD}{\ensuremath{\mathbf{D}}\xspace}
\nc{\bE}{\ensuremath{\mathbf{E}}\xspace}
\nc{\bF}{\ensuremath{\mathbf{F}}\xspace}
\nc{\bG}{\ensuremath{\mathbf{G}}\xspace}
\nc{\bH}{\ensuremath{\mathbf{H}}\xspace}
\nc{\bI}{\ensuremath{\mathbf{I}}\xspace}
\nc{\bJ}{\ensuremath{\mathbf{J}}\xspace}
\nc{\bK}{\ensuremath{\mathbf{K}}\xspace}
\nc{\bL}{\ensuremath{\mathbf{L}}\xspace}
\nc{\bM}{\ensuremath{\mathbf{M}}\xspace}
\nc{\bN}{\ensuremath{\mathbf{N}}\xspace}
\nc{\bO}{\ensuremath{\mathbf{O}}\xspace}
\nc{\bP}{\ensuremath{\mathbf{P}}\xspace}
\nc{\bQ}{\ensuremath{\mathbf{Q}}\xspace}
\nc{\bR}{\ensuremath{\mathbf{R}}\xspace}
\nc{\bS}{\ensuremath{\mathbf{S}}\xspace}
\nc{\bT}{\ensuremath{\mathbf{T}}\xspace}
\nc{\bU}{\ensuremath{\mathbf{U}}\xspace}
\nc{\bV}{\ensuremath{\mathbf{V}}\xspace}
\nc{\bW}{\ensuremath{\mathbf{W}}\xspace}
\nc{\bX}{\ensuremath{\mathbf{X}}\xspace}
\nc{\bY}{\ensuremath{\mathbf{Y}}\xspace}
\nc{\bZ}{\ensuremath{\mathbf{Z}}\xspace}

\nc{\dA}{\ensuremath{\mathds{A}}\xspace}
\nc{\dB}{\ensuremath{\mathds{B}}\xspace}
\nc{\dC}{\ensuremath{\mathds{C}}\xspace}
\nc{\dD}{\ensuremath{\mathds{D}}\xspace}
\nc{\dE}{\ensuremath{\mathds{E}}\xspace}
\nc{\dF}{\ensuremath{\mathds{F}}\xspace}
\nc{\dG}{\ensuremath{\mathds{G}}\xspace}
\nc{\dH}{\ensuremath{\mathds{H}}\xspace}
\nc{\dI}{\ensuremath{\mathds{I}}\xspace}
\nc{\dJ}{\ensuremath{\mathds{J}}\xspace}
\nc{\dK}{\ensuremath{\mathds{K}}\xspace}
\nc{\dL}{\ensuremath{\mathds{L}}\xspace}
\nc{\dM}{\ensuremath{\mathds{M}}\xspace}
\nc{\dN}{\ensuremath{\mathds{N}}\xspace}
\nc{\dO}{\ensuremath{\mathds{O}}\xspace}
\nc{\dP}{\ensuremath{\mathds{P}}\xspace}
\nc{\dQ}{\ensuremath{\mathds{Q}}\xspace}
\nc{\dR}{\ensuremath{\mathds{R}}\xspace}
\nc{\dS}{\ensuremath{\mathds{S}}\xspace}
\nc{\dT}{\ensuremath{\mathds{T}}\xspace}
\nc{\dU}{\ensuremath{\mathds{U}}\xspace}
\nc{\dV}{\ensuremath{\mathds{V}}\xspace}
\nc{\dW}{\ensuremath{\mathds{W}}\xspace}
\nc{\dX}{\ensuremath{\mathds{X}}\xspace}
\nc{\dY}{\ensuremath{\mathds{Y}}\xspace}
\nc{\dZ}{\ensuremath{\mathds{Z}}\xspace}

\nc{\bbA}{\ensuremath{\mathbb{A}}\xspace}
\nc{\bbB}{\ensuremath{\mathbb{B}}\xspace}
\nc{\bbC}{\ensuremath{\mathbb{C}}\xspace}
\nc{\bbD}{\ensuremath{\mathbb{D}}\xspace}
\nc{\bbE}{\ensuremath{\mathbb{E}}\xspace}
\nc{\bbF}{\ensuremath{\mathbb{F}}\xspace}
\nc{\bbG}{\ensuremath{\mathbb{G}}\xspace}
\nc{\bbH}{\ensuremath{\mathbb{H}}\xspace}
\nc{\bbI}{\ensuremath{\mathbb{I}}\xspace}
\nc{\bbJ}{\ensuremath{\mathbb{J}}\xspace}
\nc{\bbK}{\ensuremath{\mathbb{K}}\xspace}
\nc{\bbL}{\ensuremath{\mathbb{L}}\xspace}
\nc{\bbM}{\ensuremath{\mathbb{M}}\xspace}
\nc{\bbN}{\ensuremath{\mathbb{N}}\xspace}
\nc{\bbO}{\ensuremath{\mathbb{O}}\xspace}
\nc{\bbP}{\ensuremath{\mathbb{P}}\xspace}
\nc{\bbQ}{\ensuremath{\mathbb{Q}}\xspace}
\nc{\bbR}{\ensuremath{\mathbb{R}}\xspace}
\nc{\bbS}{\ensuremath{\mathbb{S}}\xspace}
\nc{\bbT}{\ensuremath{\mathbb{T}}\xspace}
\nc{\bbU}{\ensuremath{\mathbb{U}}\xspace}
\nc{\bbV}{\ensuremath{\mathbb{V}}\xspace}
\nc{\bbW}{\ensuremath{\mathbb{W}}\xspace}
\nc{\bbX}{\ensuremath{\mathbb{X}}\xspace}
\nc{\bbY}{\ensuremath{\mathbb{Y}}\xspace}
\nc{\bbZ}{\ensuremath{\mathbb{Z}}\xspace}

\nc{\mrm}[1]{\ensuremath{\mathrm{#1}}\xspace}
\nc{\mbf}[1]{\ensuremath{\mathbf{#1}}\xspace}
\nc{\mcal}[1]{\ensuremath{\mathcal{#1}}\xspace}
\nc{\msc}[1]{\ensuremath{\mathscr{#1}}\xspace}

\renc{\bar}[1]{\overline{#1}}

\let\S\relax

\nc{\sub}{\subset}
\nc{\too}{\longrightarrow}
\nc{\hook}{\hookrightarrow}
\nc*{\hooklongrightarrow}{\ensuremath{\lhook\joinrel\relbar\joinrel\rightarrow}}
\nc{\hooklong}{\hooklongrightarrow}
\nc{\twoheadlongrightarrow}{\relbar\joinrel\twoheadrightarrow}
\nc{\shiso}{\approx}
\nc{\isoto}{\xrightarrow{\sim}}
\nc{\isofrom}{\xleftarrow{\sim}}
\renc{\ge}{\geqslant}
\renc{\le}{\leqslant}
\renc{\geq}{\geqslant}
\renc{\leq}{\leqslant}

\nc{\id}{\mathrm{id}}
\DeclareMathOperator{\Ker}{\mathrm{Ker}}
\DeclareMathOperator{\Coker}{\mathrm{Coker}}

\let\Im\relax
\DeclareMathOperator{\Im}{\mathrm{Im}}

\DeclareMathOperator{\Hom}{\mathrm{Hom}}
\nc{\uHom}{\underline{\smash{\Hom}}}

\DeclareMathOperator{\End}{\mathrm{End}}

\nc{\Pre}{\mathrm{PSh}{}}
\nc{\Shv}{\mathrm{Shv}{}}
\nc{\uEnd}{\underline{\smash{\End}}}

\renc{\lim}{\operatorname*{lim}}
\nc{\colim}{\operatorname*{colim}}
\nc{\Cofib}{\on{Cofib}}
\nc{\Fib}{\on{Fib}}
\nc{\initial}{\varnothing}
\nc{\op}{\mathrm{op}}

\renc{\coprod}{\sqcup}

\nc{\bDelta}{\mbf{\Delta}}
\nc{\DM}{\mbf{DM}}
\nc{\eff}{\mathrm{eff}}
\nc{\veff}{\mathrm{veff}}
\nc{\cyc}{{\mrm{cyc}}}
\nc{\corr}{{\on{corr}}}
\nc{\ft}{\mrm{ft}}
\nc{\flf}{\mrm{flf}}
\nc{\fet}{{\mrm{f\acute et}}}
\nc{\fsyn}{{\mrm{fsyn}}}
\nc{\syn}{{\mrm{syn}}}
\nc{\lci}{{\mrm{lci}}}
\nc{\Perf}{\mbf{Perf}}
\nc{\perf}{\mrm{perf}}
\nc{\oblv}{\mrm{oblv}}
\nc{\exact}{\on{exact}}

\nc{\F}{{\on{F}}}
\nc{\clopen}{{\mrm{clopen}}}
\nc{\B}{\mrm{B}}
\nc{\D}{\mrm{D}}
\nc{\Fin}{\on{Fin}}
\nc{\fin}{\mrm{fin}}
\nc{\Cut}{\on{Cut}}
\nc{\Cart}{\on{Cart}}
\nc{\pairs}{\mathsf{pairs}}
\nc{\Pairs}{\mathrm{Pair}}
\nc{\Trip}{\mathrm{Trip}}
\nc{\Lab}{\mathrm{Lab}}
\nc{\SL}{\mathrm{SL}}
\nc{\coCart}{\mathrm{coCart}}
\nc{\RKE}{\mathrm{RKE}}
\nc{\strict}{\mathrm{strict}}
\nc{\Emb}{\mathrm{Emb}}
\nc{\Split}{\mathrm{Split}}
\nc{\Set}{\mathrm{Set}}
\nc{\sSets}{\mathrm{sSets}}
\nc{\pb}{\mathrm{pb}}
\nc{\fib}{\mathrm{fib}}
\nc{\diff}{\mrm{diff}}
\nc{\gp}{\mrm{gp}}
\nc{\map}{\mrm{map}}

\nc{\mgp}{\mrm{mot-gp}}
\nc{\FSyn}{\mrm{FSyn}}
\nc{\FEt}{\mrm{FEt}}
\nc{\Spc}{\mrm{Spc}}
\nc{\Ob}{\mrm{Ob}}
\nc{\Spt}{\mrm{Spt}}
\nc{\T}{\bT}
\nc{\suspinf}{\Sigma^\infty}
\nc{\h}{\mrm{h}}
\nc{\uhom}{\underline{\mathrm{Hom}}}
\nc{\umap}{\underline{\mathrm{Maps}}}
\renc{\H}{\bH}
\nc{\Einfty}{{\sE_\infty}}
\nc{\Eone}{{\sE_1}}
\nc{\Stab}{\mrm{Stab}}
\nc{\lax}{{\mrm{lax}}}
\nc{\cocart}{{\mrm{cocart}}}
\nc{\Sch}{\on{Sch}}
\nc{\Fr}{\on{Fr}}
\nc{\A}{\mathbf{A}}
\nc{\N}{\mathbf{N}}
\nc{\Z}{\mathbf{Z}}
\nc{\Q}{\mathbf{Q}}
\nc{\Oo}{\mathcal{O}} 
\nc{\Fscr}{\mathcal{F}}
\nc{\Gscr}{\mathcal{G}}
\nc{\Ll}{\mathcal{L}} 
\nc{\Mm}{\mathcal{M}} 
\nc{\mm}{\mathrm{m}} 
\nc{\K}{\mrm{K}} 
\nc{\W}{\mrm{W}} 
\nc{\red}{{\on{red}}}
\nc{\Voev}{{\on{Voev}}}
\nc{\Corr}{\mrm{Corr}}
\nc{\Span}{\mathbf{Corr}}
\nc{\Gap}{\mrm{Gap}}
\nc{\Corrfr}{\Corr^{\fr}}
\nc{\Corrvfr}{\Corr^{\Vfr}}
\nc{\Spec}{\on{Spec}}
\nc{\Sm}{\on{Sm}}
\nc{\Gm}{\mathbf{G}_{\on{m}}}
\renc{\P}{\bP}
\nc{\nis}{\mathrm{nis}}
\nc{\zar}{\mathrm{zar}}
\nc{\et}{\mathrm{\acute et}}
\nc{\all}{\mathrm{all}}
\nc{\fold}{\mathrm{fold}}
\nc{\Fun}{\mathrm{Fun}}
\nc{\Ho}{\mathrm{Ho}}
\nc{\Segal}{\mathrm{Segal}}
\nc{\Mon}{\mrm{Mon}{}}
\nc{\Ab}{\mrm{Ab}}
\nc{\Sh}{\on{Sh}}
\nc{\M}{\mrm{M}}
\nc{\Lhtp}{L_{\A^1}}
\nc{\Lmot}{L_{\mrm{mot}}}
\nc{\mot}{\mrm{mot}}
\nc{\SH}{\mbf{SH}}
\nc{\RR}{\mbf{R}}
\nc{\CC}{\mbf{C}}
\nc{\Mod}{\mbf{Mod}}
\nc{\QCoh}{\mbf{QCoh}}
\nc{\MonUnit}{\mbf{1}}
\nc{\1}{\mbf{1}}
\nc{\tr}{\on{tr}}
\nc{\cotr}{\mrm{cotr}}
\nc{\vop}{\mrm{vop}}
\nc{\fr}{{\on{fr}}}
\nc{\Ar}{\mrm{Ar}}
\nc{\Vfr}{\on{Vfr}}
\nc{\frdiff}{{\on{frdiff}}}
\nc{\frGys}{\on{frGys}}
\nc{\SHfr}{\SH^{\fr}}
\nc{\SHfrdiff}{\SH^{\frdiff}}
\nc{\SHfrGys}{\SH^{\frGys}}
\nc{\InftyCat}{(\mathrm{\infty,1)\textnormal{-}Cat}}
\nc{\TriCat}{\mathrm{TriCat}}
\nc{\oneCat}{\mathrm{1\textnormal{-}Cat}}
\nc{\Cat}{\mathrm{Cat}}
\nc{\Th}{\on{Th}}

\nc{\CMon}{\mrm{CMon}{}}
\nc{\CAlg}{\mrm{CAlg}{}}
\nc{\MGL}{\mrm{MGL}}
\nc{\Seg}{\mrm{Seg}{}}
\nc{\GW}{\mrm{GW}{}}
\nc{\Tw}{\mrm{Tw}}
\nc{\sslash}{/\mkern-6mu/}
\nc{\PrL}{\mrm{Pr}^\mrm{L}}
\nc{\PrR}{\mrm{Pr}^\mrm{R}}
\nc{\pr}{\mrm{pr}}
\let\phi\varphi

\nc\efr{\mrm{efr}}
\nc\nfr{\mrm{nfr}}
\nc\dfr{\mrm{fr}}
\nc\tfr{\mrm{tfr}}
\nc\Vect{\mrm{Vect}}
\nc\sVect{\mrm{sVect}}
\nc{\fix}{\mrm{fix}}
\nc{\ho}{\mrm{h}}
\nc\Mfd{\mrm{Mfd}}
\nc{\PSh}{\mrm{PSh}}
\nc{\hzmw}{H \tilde\Z{}}
\nc{\Cor}{\mrm{Cor}{}}
\nc{\cormw}{\mrm{\widetilde{Cor}}{}}
\nc{\Chw}{\mrm{\widetilde{CH}}{}}
\nc{\Ex}{\mrm{Ex}}
\nc{\BM}{\mrm{BM}}

\nc{\Pic}{\mrm{Pic}}
\nc{\Br}{\mrm{Br}}
\nc{\pur}{\mathfrak p}
\nc{\angles}[1]{\langle #1\rangle}
\nc{\inv}[1]{[\tfrac{1}{#1}]}
\nc{\pinv}{\inv{p}}
\nc{\cinv}{\inv{p}}
\nc{\Sph}{\on{Sph}}
\nc{\KGL}{\mrm{KGL}}
\nc{\KH}{\mrm{KH}}
\nc{\Flag}{\mrm{Flag}}
\nc{\Pro}{\mrm{Pro}}
\nc{\Frac}{\mrm{Frac}}

\nc{\arc}{\mrm{arc}}
\nc{\rarc}{\mrm{rarc}}
\nc{\cdarc}{\mrm{cdarc}}
\nc{\vv}{\mrm{v}}
\nc{\rv}{\mrm{rv}}
\nc{\cdv}{\mrm{cdv}}
\nc{\hh}{\mrm{h}}
\nc{\cdh}{\mrm{cdh}}
\nc{\rh}{\mathrm{rh}}
\nc{\Et}{\mathrm{Et}}
\nc{\Nis}{\mathrm{Nis}}
\nc{\Zar}{\mathrm{Zar}}
\nc{\cdp}{\mathrm{cdp}}

\nc{\RZ}{\mathrm{RZ}}
\nc{\qcqs}{\mathrm{qcqs}}
\nc{\aff}{\mathrm{aff}}
\nc{\cl}{\mathrm{cl}}
\nc{\Val}{\mathrm{Val}}
\nc{\GFin}{\mathrm{GFin}{}}
\nc{\Proj}{\mathrm{Proj}}

\nc{\inftyCat}{\term{$\infty$-category}}
\nc{\inftyCats}{\term{$\infty$-categories}}

\nc{\inftyOneCat}{\term{$(\infty,1)$-category}}
\nc{\inftyOneCats}{\term{$(\infty,1)$-categories}}

\nc{\inftyGrpd}{\term{$\infty$-groupoid}}
\nc{\inftyGrpds}{\term{$\infty$-groupoids}}

\nc{\inftyTop}{\term{$\infty$-topos}}
\nc{\inftyTops}{\term{$\infty$-toposes}}

\nc{\inftyTwoCat}{\term{$(\infty,2)$-category}}
\nc{\inftyTwoCats}{\term{$(\infty,2)$-categories}}

\DeclareMathOperator{\Ann}{Ann}
\DeclareMathOperator{\Fit}{Fit}
\DeclareMathOperator{\Ext}{Ext}

\title{An explicit self-duality}

\numberwithin{cor}{section}

\author[N. Kuhn]{Nikolas Kuhn}

\author[D. Mallory]{Devlin Mallory}

\author[V. Thatte]{Vaidehee Thatte}

\author[K. Wickelgren]{Kirsten Wickelgren}

\begin{document}

\def\comp{\wedge}

\begin{abstract}
We provide an exposition of the canonical self-duality associated to a presentation of a finite, flat, complete intersection over a Noetherian ring, following work of Scheja and Storch.
\end{abstract}

\maketitle
\section{Introduction}

Consider a finite ring map $A\to B$ and assume that $A$ is Noetherian. Coherent duality for proper morphisms provides a functor $f^!: D(\Spec A)\to D(\Spec B)$ on derived categories. The finiteness assumption on $f$ implies that $f^! A $ is isomorphic to the sheaf on $B$ associated to $\Hom_A(B,A)$. See for example \cite[Ideal Theorem and III \S 6]{MR0222093}. If we assume moreover that $f :\Spec B\to \Spec A$ is a local complete intersection morphism, then $f^! A$ is locally free \cite[0B6V]{stacks}. We thus obtain an isomorphism 
\begin{equation}\label{HomBA=Biso}
\Hom_A(B,A)\cong B
\end{equation} of $B$-modules under additional hypotheses, for example if we assume that $B$ local.

An explicit presentation of $B$ as \begin{equation}\label{BpresIntro}B=A[x_1,\dots,x_n]/(f_1,\dots,f_n)\end{equation} provides a {\em canonical} choice for the isomorphism~\eqref{HomBA=Biso}.
In this expository paper, we follow the approach of \cite{SchejaStorch} to construct this canonical isomorphism for $B$ a finite, flat $A$-algebra equipped with a presentation \eqref{BpresIntro}.

The approach is as follows:
Consider the ideals 
$$
(f_1\otimes 1-1\otimes f_1,\dots,f_n\otimes 1-1\otimes f_n)\subset
(x_1\otimes 1-1\otimes x_1,\dots,x_n\otimes 1-1\otimes x_n)
$$
of $A[x_1,\dots,x_n] \otimes A[x_1,\dots,x_n]$.
One writes
$$
f_j\otimes 1 -1\otimes f_j = \sum a_{ij} (x_1\otimes 1-1\otimes x_1,\dots,x_n\otimes 1-1\otimes x_n).
$$
and defines the element $\Delta \in B\otimes_A B$ as the image of $\det(a_{ij})$ under the morphism 
$A[x_1,\dots,x_n]\otimes A[x_1,\dots,x_n]\to B\otimes_A B$. This is shown to be independent of the choice of $a_{ij}$.
There is a canonical $A$-module morphism
$$
\chi: B\otimes _A B \to \Hom_A(\Hom_A(B,A),B).
$$
Let $I$ denote the kernel of multiplication $B \otimes_A B \to B$, or in other words the image of $(x_1\otimes 1-1\otimes x_1,\dots,x_n\otimes 1-1\otimes x_n)$. One checks that $\chi$ restricts to an isomorphism
$$
 \chi: \Ann_{B\otimes _A B} I
 \to \Hom_B(\Hom_A(B,A),B) 
$$
of $B$-modules and identifies the annihilator as
$ \Ann_{B\otimes _A B} I \cong \Delta$. Finally, one shows that
$$\chi(\Delta) =: \Theta \in \Hom_B(\Hom_A(B,A),B) $$
provides the desired isomorphism of $B$-modules $\Theta: \Hom_A(B,A)\to B$ guaranteed by the general theory of coherent duality.

Our arguments largely follow the outline of \cite{SchejaStorch}, although we make more use of Koszul homology in some proofs than the original did, and provide a self-contained proof of Lemma \ref{lem:1.4}; the goal in large part is to provide an English reference for this material. See also \cite[Appendices H and I]{MR2156630}.

\begin{rem}
One motivation for providing an explicit description of this isomorphism is to describe the resulting $A$-valued bilinear form on $B$. This form is defined via
$$
\langle b,c\rangle \mapsto \Theta^{-1}(b)(c) = \eta (bc) \in A,
$$
where $\eta = \Theta^{-1}(1)$. The form $\langle -, - \rangle$ has been used to give a notion of degree \cite{MR467800} \cite[some remaining questions (3)]{MR494226}.  For example, it computes the local $\mathbb{A}^1$-Brouwer degree of Morel \cite{MR3909901}.

\end{rem}

\subsection{Acknowledgements} Kirsten Wickelgren was partially supported by NSF CAREER DMS 2001890 and NSF DMS 2103838.

\section{Commutative Algebra Preliminaries}

\begin{lem}\label{lem:1.2} \cite[1.2]{SchejaStorch} Let $A$ be a noetherian ring and suppose that $f_1, \dots, f_n$ and  $g_1, \dots, g_n$ are sequences satisfying the following hypotheses:

\begin{enumerate}
\item[(i)] $\mathfrak{b} = (g_1, \dots, g_n) \subset \mathfrak{a} = (f_1, \dots f_n)$
\item[(ii)] If $\mathfrak{p}$ is a prime such that $\mathfrak{a} \subset \mathfrak{p}$, then 
the sequence $f_1, \dots, f_n$ is a regular sequence in $A_\p$, as is 
$g_1, \dots, g_n$.
\end{enumerate}
Write $ \displaystyle g_i = \sum^n_{i=1} a_{ij}f_j $, and let $(a_{ij})$ be the resulting matrix of coefficients.
\[
 \Delta:= \det\left(a_{ij}\right).
\]
Define $\overline{\Delta}$ to be the image of $\Delta$ under the map $A \rightarrow A/\mathfrak{b}$. Then:

\begin{enumerate}
\item[(a)] The element $\overline{\Delta}$ is independent of the choices of $a_{ij}$.
\item[(b)] We have an equality (of $A/\mathfrak{b}$-ideals):
\[
( \overline{\Delta}) = \Fit_{A/\mathfrak{b}}(\mathfrak{a}/\mathfrak{b}),
\]
where $\Fit $ denotes the 0-th Fitting ideal.
\item[(c)] We have an equality of ideals:
\[
( \overline{\Delta}) = \Ann_{A/\mathfrak{b}}(\mathfrak{a}/\mathfrak{b}),
\]
and
\[
\mathfrak{a}/\mathfrak{b} =  \Ann_{A/\mathfrak{b}}(\overline{\Delta}).
\]
\end{enumerate}
\end{lem}

\begin{rem}
We comment on condition (ii). If $(A,\mathfrak p)$ is a local ring and $\mathfrak a \subset \mathfrak p$, then condition (ii) is equivalent to asking that $f_1, \dots, f_n$ and  $g_1, \dots, g_n$ are regular sequences.  In general, condition (ii) asks only that they are regular sequences after localizing at primes containing $\mathfrak a$ (e.g., they may not be regular sequences on $A$).
\end{rem}

\begin{proof} First, we may assume that $A$ is a local ring and each of the $f_i$'s and $g_i$'s are in the maximal ideal $\mathfrak{m}$. 

(a): Write $g_i = \sum^{n}_{i=1} b_{ij}f_j$. We want to show that $\det(a_{ij}) - \det(b_{ij})$ is in $\mathfrak{b}$. It suffices to consider the case where $a_{ij} = b_{ij}$ for all $j$ and for $i=1, \dots, n-1$, as this allows us to change the presentation of one $g_i$ at a time, and thus all of them.
 Define
\[
c_{ij}=\begin{cases}
a_{ij} = b_{ij} & i = 1, \dots, n-1\\
a_{ij} - b_{ij} & i = n,
\end{cases}
\]
By cofactor expansion along the $j$-th row, we have that 
\[
\det(a_{ij}) - \det(b_{ij}) = \det(c_{ij}). 
\]
But now 
\[
\bigl(c_{ij}\bigr) \cdot \left(\begin{array} {c} f_1 \\ f_2 \\ \dots \\ f_{n-1} \\ f_n \end{array} \right) =  \left(\begin{array} {c} g_1 \\ g_2 \\ \dots \\ g_{n-1} \\0 \end{array} \right)
\]
By Cramer's rule, for all $k = 1, \dots, n$ we have that
\[
\det(c_{ij}) \cdot f_k \in (g_1, \dots, g_{n-1}),
\]
which means
\[
\det(c_{ij}) \cdot \mathfrak{a} \in (g_1, \dots, g_{n-1}).
\]
But $g_n \in \mathfrak{a}$ and hence 
\[
\det(c_{ij}) \cdot g_n \in (g_1, \dots, g_{n-1}),
\]
which means that $\det(c_{ij}) \in (g_1, \dots, g_n) = \mathfrak{b}$ since $g_1, \dots, g_n$ is a regular sequence. 

(b): First observe that 
\[
\Fit_A(\mathfrak{a}/\mathfrak{b}) \mod \mathfrak{b} = \Fit_{A/\mathfrak{b}}(\mathfrak{a}/\mathfrak{b}).
\]
Therefore, to prove the claim, it suffices to prove that
\[
\Fit_A(\mathfrak{a}/ \mathfrak b) = \Delta + I, 
\]
where $I \subset \mathfrak{b}$. 

To prove this claim, note that the Fitting ideal of the $A$-module $\mathfrak a / \mathfrak b$ is computed by a presentation:
\[
A^{\oplus n} \oplus A^{\oplus  {n \choose 2}} \xrightarrow{T} A^{\oplus n} \rightarrow \mathfrak{a}/\mathfrak{b} \rightarrow 0,
\]
where $T$ is given by:
\[
(a_{ij}) \times d_2^{\mrm{Kosz}}.
\]
In other words, the matrix of $T$ has the first $n$-columns are just given by $a_{ij}$ and, the last ${n \choose 2}$ columns are composed of the usual Koszul relations among the $f_i$. (Note that the sequence $f_1,\ldots,f_n$ is regular in our local ring, so the corresponding Koszul complex produces a resolution of $\mathfrak{a}$ \cite[062F]{stacks}.) 

Now, the Fitting ideal is given by the $n \times n$-minors of the matrix of $T$. The first minor is $\Delta$. If $\Delta'$ is another $n\times n$ minor, then it is the determinant of a matrix $T'$, which is composed of some $r$ columns of $(a_{ij})$ and $n-r$ columns of $d_1^{\mrm{Kosz}}$; without loss of generality we may assume $T'$ contains the first $r$ columns of $(a_{ij})$ (if not, simply reorder the $g_i$, using that the ring $A$ is local and thus regularity of the sequence of $g_i$ preserved). Applying $T'$ to $(f_k)$ we get
$$
\bigl(T' \bigr)
 \left(\begin{array} {c} f_1 \\ f_2 \\ \vdots \\ f_{n}  \end{array} \right)
=
 \left(\begin{array} {c} g_1 \\ \vdots\\g_r \\0\\ \vdots \\ 0  \end{array} \right)
$$
 
We again conclude that $\Delta' f_i =\det(T')f_i \in \mathfrak b$ for each $i=1,\dots,n$. Thus,
\[
\Delta' \cdot \mathfrak{a} \in (g_1, \dots, g_{{n-1}}), 
\]
and in particular
\[
\Delta' \cdot g_n \in (g_1, \dots, g_{{n-1}}),
\]
which by regularity of the $g_i$ means that $\Delta' \in \mathfrak{b}$ and thus $\Fit_A(\mathfrak a/ \mathfrak b)=\Delta +I$ with $I\subset \mathfrak b$.

(c): First, we claim that we have an isomorphism:
\[
\Ann_{A/\mathfrak{b}}(\mathfrak{a}/\mathfrak{b}) \cong \mrm{Tor}^A_n(A/\mathfrak{b}, A/\mathfrak{a}). 
\] We will abbreviate $\mrm{Tor}^A_j$ by $\mrm{Tor}_j$ and $\otimes_A$ by $\otimes$ in what follows.
To prove this, we deploy the Koszul complex. (As noted above, a regular sequence is Koszul-regular by \cite[062F]{stacks}.) We thus have a quasi-isomorphism.
\[
K_{\bullet}(f_1, \dots, f_n) \simeq A/\mathfrak{a} 
\]
Therefore the Tor group above is computed as the 
kernel of $1 \otimes d_n^{\mrm{Kosz}}$ in the
complex $A/\mathfrak{b} \otimes K_{\bullet}(f_1, \dots, f_n)$:
\[
0 \rightarrow A/\mathfrak{b} \xrightarrow{(f_1, \dots, f_n)} (A/\mathfrak{b})^{\oplus n}.
\]
Indeed, the cohomology of this small complex is the desired annihilator and thus we obtain the desired isomorphism. 

On the other hand, we claim that $\mrm{Tor}_n(A/\mathfrak{a}, A/\mathfrak{b}) \cong \Delta \cdot A/\mathfrak{b}$. To see this note that we have a short exact sequence of $A$-modules:
\[
0 \rightarrow \mathfrak{a}/\mathfrak{b} \rightarrow A/\mathfrak{b} \rightarrow A/\mathfrak{a} \rightarrow 0.
\]
We claim that the induced long exact sequence splits into short exact sequences for $j \geq 1$
\[
0 \rightarrow \mrm{Tor}_j(A/\mathfrak{b}, \mathfrak{a}/\mathfrak{b}) \rightarrow \mrm{Tor}_j(A/\mathfrak{b}, A/\mathfrak{b}) \rightarrow \mrm{Tor}_j(A/\mathfrak{b}, A/\mathfrak{a}) \rightarrow 0
\]
Indeed, via the Koszul complex for $A/\mathfrak{b}$, we see that for $j \geq 1$:
\begin{equation} \label{eq:tor-ab}
 \mrm{Tor}_j(A/\mathfrak{b}, \mathfrak{a}/\mathfrak{b}) \cong (\mathfrak{a}/\mathfrak{b})^{n \choose j} \qquad  \mrm{Tor}_j(A/\mathfrak{b}, A/\mathfrak{b}) \cong( A/\mathfrak{b})^{n \choose j}, 
\end{equation}
and the map $ \mrm{Tor}_j(A/\mathfrak{b}, \mathfrak{a}/\mathfrak{b}) \rightarrow  \mrm{Tor}_j(A/\mathfrak{b}, A/\mathfrak{b})$ is identified with the direct sum of copies of the injection $\mathfrak{a}/\mathfrak{b} \hookrightarrow A/\mathfrak{b}$. To conclude, the functoriality of the Koszul complex \cite[0624]{stacks} yields a morphism of complexes
\[
A/\mathfrak{b} \otimes K_{\bullet}(g_1, \dots, g_n) \rightarrow A/\mathfrak{b} \otimes K_{\bullet}(f_1, \dots, f_n);
\]
where the left end is as follows:
\begin{equation} \label{eq:kosz}
\begin{tikzcd}
A/\mathfrak{b} \ar{r}{0} \ar[swap]{d}{\overline{\Delta}} & (A/\mathfrak{b})^{\oplus n} \ar{d}\\
A/\mathfrak{b} \ar{r}{(f_1, \dots, f_n)} & (A/\mathfrak{b})^{\oplus n}.
\end{tikzcd}
\end{equation}
Since the map $\mrm{Tor}_j(A/\mathfrak{b}, A/\mathfrak{b}) \rightarrow \mrm{Tor}_j(A/\mathfrak{b}, A/\mathfrak{a})$ is a surjection, we conclude that
\[
 \mrm{Tor}_n(A/\mathfrak{b}, A/\mathfrak{a}) \cong \Im(\overline{\Delta}) \cong  \Delta \cdot A/\mathfrak{b}
\] 
as desired. 

For the second claim, note that the ideal $\Ann_{A/\mathfrak{b}}(\overline{\Delta})$ is obtained as the kernel of the left vertical map in~\eqref{eq:kosz}, and is thus isomorphic to
$\mrm{Tor}_n(A/\mathfrak{b},\mathfrak{a}/\mathfrak{b})$, which we already know is isomorphic to $\mathfrak{a}/\mathfrak{b}$ by~\eqref{eq:tor-ab}.

\end{proof}

A module $M$ over a ring $R$ is said to be reflexive if the natural map $R \to \Hom_R(\Hom_R(M,R),R)$ is an isomorphism \cite[0AUY]{stacks}. A form of the following lemma is in the stacks project (\cite[0AVA]{stacks}), but assumes that $A$ is integral and that $A=B$.  The following is \cite[1.3]{SchejaStorch}.
\begin{lem}\label{lemref}
Let $A$ be a Noetherian ring and $B$ a finite flat $A$-algebra. A finite $B$-module $M$ is reflexive if and only if the following conditions hold:
\begin{enumerate}[label=(\roman*)]
	\item If $\mathfrak{p}\subset A$ is a prime ideal with $\operatorname{depth} A_{\mathfrak{p}}\leq 1$, then $M_{\mathfrak{p}}$ is a reflexive $B_{\mathfrak{p}}$-module. \label{lemref1}
	\item If $\mathfrak{p}\subset A$ is a prime ideal with $\operatorname{depth} A_{\mathfrak{p}}\geq 2$, then $\operatorname{depth}_{A_{\mathfrak{p}}}(M_{\mathfrak{p}})\geq 2$. \label{lemref2}
\end{enumerate}
\end{lem}

\begin{proof}
The property of being reflexive is preserved under any localization of $B$ \cite[0EB9]{stacks}, and can be checked locally on $B$ \cite[0AV1]{stacks}. Therefore reflexivity of $M$ implies \ref{lemref1}. Reflexivity implies \ref{lemref2}:  Any regular sequence in $A_{\mathfrak{p}}$ is a regular sequence on $B_{\mathfrak{p}}$ by flatness. Let $a_1,a_2$ be a length $2$ regular sequence on $A_{\mathfrak{p}}$. Let $N$ be any $B_{\mathfrak{p}}$-module. Then $a_1$ is a nonzerodivisor on $\Hom_{B_\mathfrak{p}}(N,B_{\mathfrak{p}})$. The cokernel of multiplication by $a_1$ is a submodule of $\Hom_{B_{\mathfrak{p}}}(N, B_{\mathfrak{p}}/a_1B_{\mathfrak{p}})$, on which $a_2$ is a nonzerodivisor. This shows the claim. 
 (Note: This is almost \cite[0AV5]{stacks}, except that we take $\Hom_B$ but want the $A$-depth.) 

Conversely, suppose $M$ is not reflexive. We assume for the sake of contradiction that properties \ref{lemref1} and \ref{lemref2} hold.  Since reflexivity can be checked locally, there is some minimal $\mathfrak{p}\subset A$ among all prime ideals of $A$ for which $M_{\mathfrak{p}}$ is a not a reflexive $B_{\mathfrak{p}}$-module. Without loss of generality, we may assume that $A$ is local with maximal ideal $\mathfrak{p}$. Since $M_{\mathfrak{p}}$ is not reflexive, we must have that $\operatorname{depth} A_{\mathfrak{p}}\geq 2$ and therefore $\operatorname{depth}_{A_{\mathfrak{p}}}(M_{\mathfrak{p}})\geq 2$. We consider the exact sequence 
\[0\to \Ker\varphi \to M\to \Hom_B(\Hom_B(M,B),B)\to \operatorname{Coker}\varphi\to 0,\]
where $\varphi$ is the canonical map to the double-dual. By assumption, $\varphi$ becomes an isomorphism after localizing at any prime of $A$ different from $\mathfrak{p}$. It follows that $\Ker\varphi$ and $\operatorname{Coker}\varphi$ have finite length. Since $\operatorname{depth}_A M\geq 1$, there exists some $x\in A$ which is a nonzerodivisor on $M$. But then $x$ is a nonzerodivisor on the finite-length module $\Ker\varphi$, which therefore must vanish. Since $\Hom_B(\Hom_B(M,B),B)$ is reflexive (as a $B$-module), it has $A$-depth $\geq 2$ by the forward implication of the lemma. The exact sequence 
\[0\to M\to \Hom_B(\Hom_B(M,B),B) \to \Coker \varphi \to 0,\] 
then shows that $\operatorname{depth}_{A_\mathfrak{p}}\Coker \varphi\geq 1$ by the standard behavior of depth in short exact sequences \cite[00LX]{stacks}. Therefore the cokernel must vanish, which shows that $M$ is reflexive.  
\end{proof}

\begin{lem}
\label{lem:1.4}
\cite[1.4]{SchejaStorch} Let $A$ be a Noetherian ring and let $B$ be a finite flat $A$-algebra. Let $M$ be a finite $B$-module, which is projective as an $A$-module. If $\Hom_B(M,B)$ is projective as a $B$-module, then $M$ is projective as a $B$-module. In particular, if $\Hom_B(M,B)$ is free, then $M$ is free. 
\end{lem}

\begin{proof}
It is enough to show that $M$ is reflexive. We are therefore reduced to checking the conditions \ref{lemref1} and \ref{lemref2} of Lemma \ref{lemref}. Clearly, \ref{lemref2} holds, since $M$ is projective over $A$. It remains to check \ref{lemref1}. We may therefore assume that $A$ is a Noetherian local ring with $\operatorname{depth} A\leq 1$, and we want to show that $M$ is projective as a $B$-module. Since $B$ is finite flat over $A$, we have depth $B_{\frak{m}}=\operatorname{depth} A$ for every maximal ideal $\frak{m}$ of $B$ \cite[0337]{stacks}.

Throughout, we will write $N^*:=\Hom_B(N,B)$ for a $B$-module $N$. Consider the map
\[\varphi:M\to M^{**}.\]
Let $C:=\Coker \varphi$. Taking a presentation
of $M$, we obtain an exact sequence
\[0\to U\to F\to M\to 0\]
with $F$ free. Consider the dual sequence
\[0\to M^*\to F^* \to U^*,\]
and let $Q:=\operatorname{Im}(F^*\to U^*)$. Since $M^*$ is projective by assumption, $Q$ has projective dimension $0$ or $1$ as a $B$-module. 

We have the commutative diagram
\begin{equation*}
\begin{tikzcd}
F\ar[r]\ar[d,"\sim"]& M \ar[d]&&\,\\
F^{**}\ar[r]& M^{**}\ar[r] &\Ext^1_B(Q,B)\ar[r] & 0
\end{tikzcd}
\end{equation*} 
with exact lower row. Since $F\to M$ is a surjection, we see that $C=\Ext^1_B(Q,B)$. 
Suppose $\operatorname{depth} A=0$. Apply the  Auslander--Buchsbaum formula \cite[090V]{stacks} to the $B_{\frak{m}}$-module $Q_{\frak{m}}$ for each maximal ideal $\frak{m}$. We find that $Q_{\frak{m}}$ has projective dimension zero, i.e., is projective. Therefore $C_{\frak{m}}=0$ and $C=0$. 

Now suppose that $\operatorname{depth} A =1$. Then $\operatorname{depth}_{B_{\frak{m}}} U_{\frak{m}}^*\geq 1$ by \cite[0AV5]{stacks}, whence $$\operatorname{depth}_{B_{\frak{m}}} Q_{\frak{m}}\geq 1$$ by  \cite[00LX]{stacks}. Again by Auslander--Buchsbaum, we find that $Q_{\frak{m}}$ is projective, and that $C=0$. 

We have shown that in any case $M\to M^{**}$ is surjective. Since $M^{**}$ is projective, this implies $M\simeq M^{**}\oplus N$ for some $B$-module $N$. It follows that $N^*=0$ and that $N$ is again free as an $A$-module.

By assumption both $M$ and $M^{**}$ are free over the local ring $A$. A surjection of finite free $A$-modules is an isomorphism if they have the same rank. To show two finite free modules have the same rank, we may localize at a minimal prime ideal $\mathfrak{q}$ of $A$, so that also $B_{\mathfrak{q}}$ is a zero-dimensional ring. 
Over the Artinian ring $B_{\mathfrak q}$, $\Hom_{B_{\mathfrak{q}}}(N_{\mathfrak{q}}, B_{\mathfrak{q}})=0$ implies $N_{\mathfrak{q}}=0$. (To see this, note that we may assume that $B$ is local, with maximal ideal $\mathfrak{m}$. Then $N_{\mathfrak{q}} \to \mathfrak{m}N_{\mathfrak{q}}$ is nonzero by Nakayama's lemma. Since $B_{\mathfrak{q}}$ has finite length, there is a nonzero element annihilated by $\mathfrak{m}$ whence a $B$-homomorphism $B/\mathfrak{m} \to B_{\mathfrak{q}}$.) Thus $M_{\mathfrak q}$ and $M^{**}_{\mathfrak q}$ have the same rank, and therefore $M\to M^{**}$ is an isomorphism.
\end{proof}

\section{The explicit isomorphism}

Recall that a ring map $A\rightarrow B$ is a \emph{relative global complete intersection}  if there exists a presentation $A[x_1, \dots, x_n]/(f_1, \dots, f_c) \cong B$, and every nonempty fiber of $\Spec B \rightarrow \Spec A$ has dimension $n-c$ \cite[00SP]{stacks}. 
Note that in this case 
the $f_i$ form a regular sequence \cite[00SV]{stacks}.

We note that a global complete intersection is flat \cite[00SW]{stacks}, and thus syntomic. We will be interested in the situation where $A\rightarrow B$ is furthermore assumed to be a \emph{finite} flat global complete intersection.

\begin{constr}\label{cons:7} Suppose that $A \rightarrow B$ is a finite flat global complete intersection. Choose a presentation
\[
A[x_1, \dots, x_n] \xrightarrow{\pi} B \cong A[x_1, \dots, x_n]/(f_1, \dots, f_n).
\]

Consider the commutative diagram
\begin{equation} \label{eq:diags}
\begin{tikzcd}
A[x_1, \dots, x_n] \otimes_A A[x_1, \dots, x_n] \ar[swap]{d}{\pi \otimes \pi} \ar{r}{m_1} & A[x_1, \dots, x_n] \ar{d}{\pi}\\
B \otimes_A B \ar{r}{m} & B,
\end{tikzcd}
\end{equation}
with $m_1,m$ the obvious multiplication maps.
We note that the elements
\[
\{ f_j \otimes 1 - 1 \otimes f_j \}_{j=1, \dots, n}
\]
are all in $\ker(m_1)$, which is generated by the $x_i\otimes 1- 1\otimes x_i$ for $i=1,\dots,n$, whence we have a relation
\[
f_j \otimes 1 - 1 \otimes f_j = \sum^{n}_{i=1} a_{ij}(x_i \otimes 1 - 1 \otimes x_i). 
\]
Define $\Delta:= (\pi \otimes \pi)(\det(a_{ij})) \in B \otimes_A B$. 
Define also $I:=\ker m$.

\end{constr}

\begin{prop}\label{lem:3.1} The following properties of $\Delta$ hold:
\begin{enumerate}
\item[(a)] The element $\Delta$ is independent of the choice of $a_{ij}$.
\item[(b)] We have an equality of $B \otimes_A B$-ideals:
\[
(\Delta) = \Fit_{B \otimes_A B}I ,
\]
\item[(c)] we have an equality of ideals
\[
(\Delta) = \Ann_{B \otimes_A B}I \qquad  \Ann_{B \otimes_A B}(\Delta) = I. 
\]
\end{enumerate}
\end{prop}

\begin{proof} 
Consider the ring map 
$$\pi\otimes 1: A[x_1,\dots,x_n]\otimes_A A[x_1,\dots,x_n]\to  B\otimes_A A[x_1,\dots,x_n]\cong B[x_1,\dots,x_n].$$
Since
\[
f_i \otimes 1 - 1 \otimes f_i = \sum^{n}_{i=1} a_{ij}(x_i \otimes 1 - 1 \otimes x_i)
\]
in $ A[x_1,\dots,x_n]\otimes_A A[x_1,\dots,x_n]$, we have that
\[
- 1 \otimes f_i = \sum^{n}_{i=1} a_{ij}(\pi(x_i) \otimes 1 - 1 \otimes x_i)
\]
in $B\otimes_A A[x_1,\dots,x_n]$.

Note that $\Delta$ is the image of $\det(a_{ij})$ under the obvious morphism $B\otimes_A A[x_1,\dots,x_n]\to B\otimes_A B$, and 
that if $\mathfrak a$ is the ideal generated by the 
$\pi(x_i) \otimes 1 - 1 \otimes x_i$ and $\mathfrak b$ the ideal generated by the
$(-1\otimes f_i)$,
then $I$ is $\mathfrak a/\mathfrak b$.
The desired properties will then follow immediately from 
applying Lemma \ref{lem:1.2} to $\mathfrak b=(-1\otimes f_i)\subset (\pi(x_i)\otimes 1-1\otimes x_i)=\mathfrak a$, once we show that the conditions of the Lemma are satisfied. It suffices to show that each is a regular sequence.

We claim that $\{ -1 \otimes f_j \} \subset 
B \otimes_A A[x_1, \dots, x_n]$ is a regular sequence. Indeed, since relative global complete intersections are flat \cite[00SW]{stacks} and regular sequences are preserved under flat morphisms, this follows 
by regularity of the $f_i$ in $A[x_1,\dots,x_n]$ and flatness of $A\rightarrow B$.
It is immediate also that $(\pi(x_i)-x_i)$ forms a regular sequence in $B[x_1,\dots,x_n]$ as well (the $\pi(x_i)$ are just elements $b_i$ of $B$, and $(x_i-b_i)$ is always a regular sequence in $B[x_1,\dots,x_n]$).

Thus, the proposition follows by Lemma~\ref{lem:1.2}.
\end{proof}

Now, retain our setup from Construction~\ref{cons:7}.
There is a canonical map of $A$-modules
\[
\chi:B \otimes_A B \rightarrow \Hom_A(\Hom_A(B, A), B) \qquad \chi(b \otimes c) = (\phi \mapsto \phi(b)c).
\]
Both $B\otimes_A B$ and $\Hom_A(\Hom_A(B, A), B)$ each carry two natural $B$-module structures:
\begin{enumerate}
\item $B$ acts on $B\otimes_A B $ as multiplication on either the left or right factor (i.e., either $a(b\otimes c) = ab\otimes c$ or $a(b\otimes c) = b\otimes ac$).
\item $B$ acts on $\Hom_A(\Hom_A(B, A), B)$ as either pre- or post-composing a homomorphism by multiplication (i.e., either 
$a\phi : \psi \mapsto\phi(a\psi)$ or 
$a\phi : \psi \mapsto a\phi(\psi)$).
\end{enumerate}

\begin{lem}\label{lem:3.2}
$\chi$ induces a $B$-module isomorphism $\Ann_{B\otimes_A B} I \cong \Hom_B(\Hom_A(B, A), B)$.
\end{lem}

\begin{proof}

We note first that this map is an isomorphism of $A$-modules, for which it suffices to check that it's bijective:
Since $B$ is a projective $A$-module we have that $B$ is canonically isomorphic to $B^{\vee \vee}$ (where we denote by ${}^\vee$ the $A$-module dual), so that we have isomorphisms of $A$-modules
$$
B\otimes_A B \cong (B^{\vee})^{\vee}\otimes_A B \cong \Hom_A(B^\vee,B) = \Hom_A(\Hom_A(B,A),B);
$$
one can check that $\chi$ is simply the composition of these canonical isomorphisms.

It's immediately checked that the morphism $\chi$ is in fact a $B$-bimodule homomorphism for the $B$-module structures of $B\otimes_A B $ and $\Hom_A(\Hom_A(B, A), B)$ given by right multiplication and post-composition. 

Now, we note the following:
\begin{enumerate}
\item The largest submodule of $B\otimes_A B$ where the two $B$-module structures agree is $\Ann_{B\otimes_A B} I$: this follows since an element $r \in B\otimes_A B$ is annihilated by all $a\otimes 1- 1\otimes a$ exactly when $(a\otimes 1)r=(1\otimes a )r$ for all $a$, which occurs exactly when the action of every $a$ on $r$ is the same under the two $B$-module structures.
\item The largest submodule of 
$\Hom_A(\Hom_A(B, A), B)$ where the two $B$-module structures agree is 
$$
\Hom_B(\Hom_A(B, A), B)
\subset
\Hom_A(\Hom_A(B, A), B);
$$
this is clear since the condition of pre- and post-multiplying by elements of $B$ being the same is exactly $B$-linearity.
\end{enumerate}

Putting this together, we have that $\chi$ induces an isomorphism \emph{of $B$-modules}
$$
\chi:
\Ann_{B\otimes _A B} I \to \Hom_B(\Hom_A(B, A), B),
$$
which was our desired claim.
\end{proof}

\begin{thm}\label{lem:3.3} The map $\chi(\Delta):\Hom_A(B,A) \rightarrow B$ is an isomorphism of $B$-modules.
\end{thm}

\begin{proof}
Applying Lemma~\ref{lem:3.1}(c) we have that 
$\Ann_{B\otimes_A B} I =\Delta (B\otimes_A B)$, and further that $\Ann_{B\otimes_A B} \Delta(B\otimes_A B) = I$. 
Thus, we have that
$$
\Ann_{B\otimes_A B} I 
=\Delta (B\otimes_A B) 
\cong \Delta(B\otimes_A B)/\Ann_{B\otimes_A B} \Delta
=\Delta(B\otimes_A B)/I
\cong m(\Delta)B.
$$
Applying Lemma~\ref{lem:3.2}, we have then that $\Hom_B(\Hom_A(B,A),B)$ is a free $B$-module with basis $\chi(\Delta)$.
Applying Lemma~\ref{lem:1.4}, this implies that $\Hom_A(B,A)$ is a free $B$-module of rank 1.
We must then have that the $B$-module homomorphism $\chi(\Delta):\Hom_A(B,A)\to B$ is an isomorphism, as desired.
\end{proof}

\bibliographystyle{alphamod}

\let\mathbb=\mathbf

{\small
\bibliography{BEKMTWSPONGE}
}

\parskip 0pt

\end{document}